\documentclass[11pt,a4paper]{article}
\usepackage[T1]{fontenc}
\usepackage{lmodern}
\usepackage[margin=27mm]{geometry}
\usepackage{amsmath,amssymb,amsthm,mathtools}
\usepackage{booktabs,microtype,xcolor}
\usepackage[colorlinks=true,linkcolor=blue!45!black,citecolor=blue!45!black,urlcolor=blue!45!black]{hyperref}
\hypersetup{
 pdftitle={The Goldbach exceptional set: an exponent of 7/12 and conditional refinements},
 pdfauthor={Janos Pintz and Ingo Althofer},
 pdfsubject={Joint research draft, 10 October 2026}}
\newtheorem{theorem}{Theorem}
\newtheorem{corollary}[theorem]{Corollary}
\newtheorem{lemma}[theorem]{Lemma}
\theoremstyle{remark}

\numberwithin{equation}{section}
\newcommand{\HH}{\mathrm H}
\newcommand{\SSS}{\mathfrak S}
\newcommand{\MM}{\mathfrak M}
\newcommand{\mm}{\mathfrak m}
\newcommand{\dd}{\,\mathrm d}
\DeclareMathOperator{\Rea}{Re}
\DeclareMathOperator{\Ima}{Im}
\title{The Goldbach exceptional set:\\
an exponent of \(7/12\) and conditional refinements}
\author{J\'anos Pintz\thanks{Keywords and phrases: Goldbach conjecture,
exceptional set for Goldbach's problem.\newline
2020 Mathematics Subject Classification: 11P32.\newline
The first author was supported by the National Research Development
and Innovation Office of Hungary, NKFIH, K133819, K147153,
Excellence 151341.} \and Ingo Alth\"ofer}
\date{}

\begin{document}
\maketitle

\begin{center}
\itshape Dedicated to the 70th birthday of Szil\'ard Gy.~R\'ev\'esz
\end{center}

\begin{abstract}
Let \(E(X)\) count the even integers between \(4\) and \(X\) that are not
sums of two primes. We combine Pintz's explicit formula for the major
arcs with the exponential-sum estimate of Maynard, Pandey, and
Radziwi{\l}{\l} to prove
\(E(X)\ll_\varepsilon X^{7/12+\varepsilon}\) whenever all Dirichlet
\(L\)-functions have a common fixed zero-free half-plane
\(\Rea s>\theta_0\), with \(\theta_0<1\).
The uniform \(7/8\) theorem reported in the OpenAI mathematics release
supplies this input.
Under the stronger hypothesis that every nontrivial Dirichlet zero
has real part at most \(\theta\), where \(1/2\leq\theta\leq2/3\), we prove
\[
 E(X)\ll_{\theta,\varepsilon}X^{g(\theta)+\varepsilon},
 \qquad
 g(\theta)=\frac{\theta(5-4\theta)}{2(2-\theta)}.
\]
Thus every fixed \(\theta<2/3\) gives an exponent strictly below \(7/12\).
The latter argument uses a classical fixed-modulus zero-density
estimate and is independent of the new exponential-sum theorem.
\end{abstract}

\section{Introduction and results}

For \(X\geq4\), define
\[
 E(X)=\#\{n\in\mathbb Z:4\leq n\leq X,\ 2\mid n,\
                   n\neq p+q\text{ for all primes }p,q\}.
\]
The primes \(p,q\) may be equal. We omit the even integer \(2\);
this changes other common definitions of \(E(X)\) by at most one.

The study of \(E(X)\) is one of the principal quantitative approaches
to the binary Goldbach problem. Montgomery and Vaughan
\cite{MV} established a power saving \(E(X)\ll X^{1-\delta}\)
for some absolute \(\delta>0\).
Pintz's explicit formula \cite{P1} allows substantially larger major
arcs while retaining the contribution of zeros close to the line
\(\Rea s=1\). Its application in \cite{P2} gives
\(E(X)<X^{0.72}\) for sufficiently large \(X\).
Zhao \cite{Zhao} subsequently obtained \(E(X)\ll X^{0.709}\).

The treatment of possible zeros near one draws on the zero-free
regions and zero-repulsion estimates of Heath-Brown \cite{HB},
and their refinement by Xylouris \cite{Xylouris}.
Pintz and Alth\"ofer, with AI assistance, developed a further series
of unpublished numerical refinements within this analytic framework,
culminating in the reported fixed exponent \(0.684\)
\cite{PA}. These computations retain dependencies between zero counts,
early zero mass, and tails, rather than maximizing them independently.
The complete numerical chain is not used in the proofs below.
Reference~\cite{PA} distinguishes the intermediate \(0.68601\)
calculation from the later \(0.684\) endpoint.

A fixed zero-free half-plane removes the possible generalized
exceptional zeros from the major-arc formula altogether, for
sufficiently large \(X\). The exceptional-set exponent is then
controlled by the minor arcs. The new estimate of Maynard, Pandey,
and Radziwi{\l}{\l} \cite{MPR} replaces the classical exponent \(4/5\)
in a prime exponential sum by \(19/24\). The corresponding
exceptional-set exponents are
\[
 2\cdot\frac45-1=\frac35,
 \qquad
 2\cdot\frac{19}{24}-1=\frac7{12}.
\]
The distinction between these two pairs of exponents is useful.

For \(1/2\leq\theta<1\), let \(\HH_\theta\) denote the assertion
\begin{equation}\label{eq:hypothesis}
 \Rea\rho\leq\theta
 \quad\text{for every nontrivial zero of every primitive
 Dirichlet \(L\)-function.}
\end{equation}
The zeta function is included as the function of conductor one.
Uniformity here means that the same \(\theta\) works for every
character and every height.

\begin{theorem}[A fixed zero-free half-plane]\label{thm:fixed}
Suppose that \(\HH_{\theta_0}\) holds for some fixed
\(1/2\leq\theta_0<1\). Then, for every \(\varepsilon>0\),
\begin{equation}\label{eq:seven}
 E(X)\ll_{\theta_0,\varepsilon}X^{7/12+\varepsilon}.
\end{equation}
\end{theorem}

\begin{corollary}[Application of the \(7/8\) theorem]\label{cor:seven}
On invoking the uniform Dirichlet \(L\)-function theorem
of OpenAI \cite[Theorem~1.1]{OAI}, one obtains
\[
 E(X)\ll_\varepsilon X^{7/12+\varepsilon}.
\]
\end{corollary}

Theorem~\ref{thm:fixed} isolates the exact implication needed from
the external zero-free theorem. A result for \(\zeta(s)\) alone
would not supply its hypothesis. This draft takes the uniform
Dirichlet statement in \cite{OAI} as an external input; it does not
reprove that result. In fact, any fixed boundary below one,
including \(11/12\), suffices for this particular application.
The value \(7/8\) affects the required lower threshold for \(X\),
but not the exponent in \eqref{eq:seven}.
The \(\varepsilon\)-loss is retained throughout: we do not assert
\(E(X)\ll X^{7/12}\) without it.

The next result addresses stronger fixed zero-free hypotheses.

\begin{theorem}[Conditional refinement]\label{thm:main}
Assume \(\HH_\theta\) for a fixed \(1/2\leq\theta\leq2/3\).
For every \(\varepsilon>0\),
\begin{equation}\label{eq:main}
 E(X)\ll_{\theta,\varepsilon}X^{g(\theta)+\varepsilon},
 \qquad
 g(\theta)=2\theta-\frac32+\frac{3(1-\theta)}{2-\theta}
 =\frac{\theta(5-4\theta)}{2(2-\theta)}.
\end{equation}
\end{theorem}

Since \(g(2/3)=7/12\) and \(g\) is strictly increasing on
\([1/2,2/3]\), every fixed \(\theta<2/3\) gives a fixed positive
saving below \(7/12\), after choosing \(\varepsilon\) small enough.
For example,
\[
 \HH_{5/8}\quad\Longrightarrow\quad
 E(X)\ll_\varepsilon X^{25/44+\varepsilon},
 \qquad
 \frac7{12}-\frac{25}{44}=\frac1{66}.
\]
This is a sufficient threshold for the method, not a claim that
\(\theta<2/3\) is necessary or optimal.

\section{Major and minor arcs}

Write \(e(t)=\exp(2\pi i t)\).
Fix a small constant \(\delta>0\), and choose a sufficiently small
constant \(\varepsilon_0>0\) as in Pintz's major-arc formula.
Put
\[
 X_1=X^{1-\varepsilon_0},
 \qquad
 S_X(\alpha)=\sum_{X_1<p\leq X}(\log p)e(p\alpha).
\]
For \(P\leq X^{4/9-\delta}\), define, on \(\mathbb R/\mathbb Z\),
\[
 \MM(P)=
 \bigcup_{q\leq P}
 \ \bigcup_{\substack{0\leq a<q\\(a,q)=1}}
 \left\{\alpha:\left|\alpha-\frac aq\right|
                    \leq\frac{P}{qX}\right\},
 \qquad
 \mm(P)=(\mathbb R/\mathbb Z)\setminus\MM(P).
\]
Distances and rational approximations on the circle are understood
modulo one. These arcs are disjoint for sufficiently large \(X\).
For \(n\in(X/2,X]\), write
\[
 R_X(n)=\sum_{\substack{p+q=n\\p,q>X_1}}\log p\log q
       =R_{\MM}(n)+R_{\mm}(n),
 \qquad
 R_{\mathcal A}(n)=\int_{\mathcal A}S_X(\alpha)^2e(-n\alpha)\dd\alpha.
\]
We use the usual Goldbach singular series
\[
 \SSS(n)=2C_2\prod_{\substack{p\mid n\\p>2}}\frac{p-1}{p-2},
 \qquad
 C_2=\prod_{p>2}\left(1-\frac1{(p-1)^2}\right)>0
 \quad (n\text{ even}).
\]

\subsection{The specialization of Pintz's explicit formula}

The following is the zero-free specialization of
\cite[Theorem~1 and equation~(2.10), arXiv version]{P1}.
For each \(P_0\leq X^{4/9-\delta}\), the formula permits a choice
\begin{equation}\label{eq:Pchoice}
 P_0X^{-\delta}\leq P\leq P_0.
\end{equation}
For fixed sufficiently large \(H\), if no primitive Dirichlet
\(L\)-function of conductor at most \(P\) has a zero in
\[
 \Rea s\geq1-\frac{H}{\log X},\qquad
 |\Ima s|\leq\sqrt X,
\]
then, uniformly for even \(n\in(X/2,X]\),
\begin{equation}\label{eq:Pintz}
 R_{\MM}(n)=\SSS(n)n+
 O_\delta\!\left(\SSS(n)Xe^{-c_\delta H}\right)
 +O_\delta(X^{1-\kappa}),
 \qquad c_\delta>0,\quad \kappa>0.
\end{equation}
Indeed, the generalized-singularity sum in that formula then contains
only the pole of \(\zeta(s)\) at one. The term from the pair of poles
is \(\SSS(n)n\). Here \(\kappa\) is a fixed positive constant;
we allow it to be smaller than \(\varepsilon_0\) to absorb any
logarithmic factors associated with the lower prime cutoff.

Under \(\HH_{\theta_0}\), the zero condition holds as soon as
\begin{equation}\label{eq:Hchoice}
 \log X>\frac{H}{1-\theta_0}.
\end{equation}
Choose \(\delta\) first, then \(H\) sufficiently large, and finally
\(X\) sufficiently large. The integral \(R_{\MM}(n)\) is real by
symmetry. Since \(\SSS(n)\geq2C_2\), equation
\eqref{eq:Pintz} yields
\begin{equation}\label{eq:majorpositive}
 R_{\MM}(n)\geq cX
 \quad (X/2<n\leq X,\ n\text{ even})
\end{equation}
for a fixed \(c>0\).
This excludes all the generalized exceptional zeros relevant to the
formula, not only a possible real Siegel zero.

\subsection{The exponential-sum estimate and the cutoff}

We use \cite[Theorem~1.1]{MPR} in the following form. If
\[
 \alpha=\frac aq+\beta,\qquad (a,q)=1,\qquad
 q\leq Y^{1/2},\qquad |\beta|\leq\frac1{qY^{1/2}},
 \qquad B_Y=\max(q,qY|\beta|),
\]
then, for every \(\eta>0\),
\begin{equation}\label{eq:MPR}
 \left|\sum_{m\leq Y}\Lambda(m)e(m\alpha)\right|
 \ll_\eta Y^\eta\left(\frac{Y}{B_Y^{1/2}}+Y^{19/24}\right).
\end{equation}
Changing a strict upper endpoint to a non-strict one contributes
only \(O(\log(2Y))\).

We record explicitly why \eqref{eq:MPR} also controls the lower
cutoff in \(S_X\). Fix \(\alpha\in\mm(P)\) and \(2\leq Y\leq X\),
and choose a Dirichlet approximation at scale \(Y\) as above.
If \(q>P\), then \(B_Y>P\geq PY/X\).
If \(q\leq P\), the fact that \(\alpha\notin\MM(P)\) implies
\(|\beta|>P/(qX)\), so again \(B_Y>PY/X\).
Consequently
\[
 \frac{Y}{\sqrt{B_Y}}\ll\sqrt{\frac{XY}{P}}
                   \leq\frac{X}{\sqrt P}.
\]
Applying \eqref{eq:MPR} at \(Y=X\) and \(Y=X_1\), and deleting
the proper prime powers, whose total von Mangoldt weight is
\(O(\sqrt X(\log X)^2)\), gives
\begin{equation}\label{eq:minor}
 \sup_{\alpha\in\mm(P)}|S_X(\alpha)|
 \ll_\eta X^\eta
 \left(\frac{X}{\sqrt P}+X^{19/24}\right).
\end{equation}
In particular, no trivial estimate of size \(X_1\) is incurred.

\section{Proof of Theorem~\ref{thm:fixed}}

Choose a fixed \(0<\delta<1/72\), put \(P_0=X^{4/9-\delta}\),
and use \eqref{eq:Pchoice}. Thus
\begin{equation}\label{eq:Psize}
 X^{4/9-2\delta}\leq P\leq X^{4/9-\delta}.
\end{equation}
Under \(\HH_{\theta_0}\), the choices described in
\eqref{eq:Hchoice} give \eqref{eq:majorpositive}.

If an even \(n\in(X/2,X]\) has no Goldbach representation, then
\(R_X(n)=0\). Hence \(|R_{\mm}(n)|=|R_{\MM}(n)|\geq cX\).
Bessel's inequality, followed by \eqref{eq:minor} and Parseval,
therefore gives
\begin{align}
 \bigl(E(X)-E(X/2)\bigr)X^2
 &\ll \int_{\mm(P)}|S_X(\alpha)|^4\dd\alpha \notag\\
 &\leq\left(\sup_{\alpha\in\mm(P)}|S_X(\alpha)|\right)^2
       \int_0^1|S_X(\alpha)|^2\dd\alpha \notag\\
 &\ll_\eta X^{2\eta}
       \left(\frac{X^2}{P}+X^{19/12}\right)X\log X.
       \label{eq:baselineBessel}
\end{align}
Here
\(\int_0^1|S_X|^2=\sum_{X_1<p\leq X}(\log p)^2\ll X\log X\).
It follows that
\begin{equation}\label{eq:baselinecount}
 E(X)-E(X/2)
 \ll_\eta X^{2\eta}\log X
       \left(\frac XP+X^{7/12}\right).
\end{equation}
By \eqref{eq:Psize},
\[
 \frac XP\leq X^{5/9+2\delta},
 \qquad
 \frac59+2\delta<\frac7{12}.
\]
Given \(\varepsilon>0\), take \(\eta>0\) sufficiently small and
sum \eqref{eq:baselinecount} over dyadic intervals.
This proves \eqref{eq:seven}. Corollary~\ref{cor:seven} follows
by setting \(\theta_0=7/8\) in the external input
\cite[Theorem~1.1]{OAI}.

This proof explains why improving the numerical boundary \(7/8\)
slightly does not itself improve the exponent \(7/12\).
Once the major-arc main term is positive, the displayed minor-arc
estimate is the limiting input. The next section uses stronger
zero-free information directly in the exponential sum.

\section{Sharper bounds under a stronger zero-free hypothesis}
\label{sec:conditional}

Throughout this section, assume \(\HH_\theta\) with
\(1/2\leq\theta\leq2/3\). We give a complete proof of
Theorem~\ref{thm:main}. Its analytic input is classical and does
not require \cite{MPR} or \cite{OAI}.

\subsection{The classical zero-density input}

For a character \(\chi\bmod q\), let \(\chi^*\) be its inducing
primitive character. Count zeros with multiplicity, and put
\[
 N_q(\sigma,U)=
 \sum_{\chi\bmod q}
 \#\{\rho=\beta+i\gamma:\ L(\rho,\chi^*)=0,
 \ 0<\beta<1,\ \beta\geq\sigma,\ |\gamma|\leq U\}.
\]
We use the classical fixed-modulus estimate
\begin{equation}\label{eq:density}
 N_q(\sigma,U)\ll
 (qU)^{D(\sigma)}\log^9(qU),
 \qquad D(\sigma)=\frac{3(1-\sigma)}{2-\sigma},
 \qquad \frac12\leq\sigma\leq\frac34,
\end{equation}
uniformly for \(q\geq1\) and \(U\geq2\).
This is explicitly recorded in Wolke \cite[eq.~(2.2.5)]{Wolke},
with a reference to Montgomery \cite[Theorem~12.1]{Montgomery}.
It sums over characters of one fixed modulus; the parameter is
\(qU\). Euler factors introduced by imprimitive characters have
zeros on \(\Rea s=0\), so they do not affect this statement.
We also use the elementary total zero count
\begin{equation}\label{eq:total}
 \sum_{\chi\bmod q}
 \#\{\rho_{\chi^*}:|\Ima\rho_{\chi^*}|\leq U\}
 \ll q(U+1)\log\bigl(q(U+2)\bigr).
\end{equation}

By the functional equations, \(\HH_\theta\) implies
\begin{equation}\label{eq:strip}
 1-\theta\leq\Rea\rho\leq\theta
 \quad\text{for all these nontrivial zeros.}
\end{equation}
The lower bound follows by applying the hypothesis also to the
conjugate character. In our range, all nontrivial zeros therefore
stay in the fixed strip \(1/3\leq\Rea\rho\leq2/3\).

\subsection{A uniform estimate near rational points}

Write \(e(t)=\exp(2\pi i t)\), and introduce
\[
 F_X(\alpha)=\sum_{m\geq1}\Lambda(m)e^{-m/X}e(m\alpha).
\]
Exponential smoothing is convenient because its Mellin transform
is a gamma factor, and its convolution gives an exact main term.
Set
\begin{equation}\label{eq:a}
 a(\theta)=\theta-\frac14+\frac{3(1-\theta)}{2(2-\theta)}.
\end{equation}

\begin{lemma}\label{lem:local}
Assume \(\HH_\theta\), with \(1/2\leq\theta\leq2/3\).
For every \(\eta>0\), uniformly when
\[
 (a,q)=1,\qquad 1\leq q\leq2\sqrt X,
 \qquad |\beta|\leq\frac{2}{q\sqrt X},
\]
one has, with \(z=X^{-1}-2\pi i\beta\),
\begin{equation}\label{eq:local}
 F_X(a/q+\beta)=\frac{\mu(q)}{\varphi(q)z}
 +O_{\theta,\eta}\bigl(X^{a(\theta)+\eta}\bigr).
\end{equation}
\end{lemma}

\begin{proof}
Put
\[
 T=1+X|\beta|,\qquad B=qT.
\]
Then \(B\leq4\sqrt X\), \(|z|\asymp T/X\), and
\(\pi/2-|\arg z|\gg1/T\).

\paragraph{Character expansion and the explicit formula.}
For \((m,q)=1\), character orthogonality gives
\[
 e(am/q)=\frac{1}{\varphi(q)}
 \sum_{\chi\bmod q}\tau(\overline\chi)\chi(a)\chi(m).
\]
Here \(|\tau(\chi)|\leq\sqrt q\), including for imprimitive
characters. Replacing \(\chi\) by \(\chi^*\) changes the
corresponding von Mangoldt sum only on powers of primes dividing
\(q\). Their contribution, and that of \((m,q)>1\), is absorbed
in \(O(\sqrt q\log^2(qX))\).

For a primitive character of conductor \(f\leq q\), Mellin
inversion and a contour shift give
\begin{equation}\label{eq:explicit-primitive}
 \sum_{m\geq1}\Lambda(m)\chi^*(m)e^{-mz}
 =\frac{\delta_{\chi^*}}{z}
 -\sum_{\rho_{\chi^*}}\Gamma(\rho_{\chi^*})z^{-\rho_{\chi^*}}
 +O_\theta\bigl(\log^2(qX)\bigr),
\end{equation}
where \(\delta_{\chi^*}=1\) for conductor one and is zero otherwise.
The power of \(z\) uses the principal logarithm.
For completeness, start with
\[
 -\frac{1}{2\pi i}\int_{(2)}
 \frac{L'}{L}(s,\chi^*)\Gamma(s)z^{-s}\dd s
\]
and move the line to \(\Rea s=-1/2\). The poles crossed are the
pole at one, the nontrivial zeros, and the possible simple or
double pole of the integrand at zero. Under \eqref{eq:strip},
the last residue is \(O_\theta(\log^2(qX))\): the regular parts
of the logarithmic derivative at zero are bounded using the
functional equation and its partial-fraction formula at one
and two. Indeed, the zeros stay a fixed distance from one and
their sum with weight \((1+|\gamma|^2)^{-1}\) is
\(O(\log(f+3))\).
On the new line the functional equation gives
\(L'/L(s,\chi^*)\ll\log(f(|\Ima s|+2))\).
Stirling's formula bounds the remaining integral by
\[
 O\!\left(|z|^{1/2}\log^2\bigl(q(T+2)\bigr)\right),
\]
which is within the error in \eqref{eq:explicit-primitive}.
Horizontal contours can be taken away from zeros; their
integrals tend to zero because \(|\arg z|<\pi/2\).
The same exponential decay makes the displayed zero sum
absolutely convergent.

Consequently, with \(c_\chi=\tau(\overline\chi)\chi(a)\),
\begin{equation}\label{eq:explicit}
 F_X(a/q+\beta)=\frac{\mu(q)}{\varphi(q)z}
 -\frac{1}{\varphi(q)}\sum_{\chi\bmod q}c_\chi
   \sum_{\rho_{\chi^*}}\Gamma(\rho_{\chi^*})z^{-\rho_{\chi^*}}
 +O_\theta\bigl(\sqrt q\log^2(qX)\bigr).
\end{equation}
The principal character supplies the main term because its
Gauss sum is \(\mu(q)\).

\paragraph{Estimating the zeros.}
For \(\rho=u+i\gamma\), Stirling's formula and
\eqref{eq:strip} give
\begin{equation}\label{eq:gamma}
 |\Gamma(\rho)z^{-\rho}|
 \ll_\theta X^u T^{-u}(1+|\gamma|)^{u-1/2}
             e^{-c|\gamma|/T}
\end{equation}
for an absolute \(c>0\).
If \(u<1/2\), the right side is at most
\(O_\theta(X^{1/2}T^{-1/2}e^{-c|\gamma|/T})\), since
\(X/T\gg\sqrt X\).
After summing by \eqref{eq:total} and multiplying by
\(\sqrt q/\varphi(q)\), all these zeros contribute
\begin{equation}\label{eq:low}
 \ll_{\theta,\eta} X^\eta X^{1/2}B^{1/2}.
\end{equation}

For \(u\geq1/2\), the bound \eqref{eq:gamma} is at most
\[
 C_\theta X^uT^{-1/2}
 (1+|\gamma|/T)^{1/6}e^{-c|\gamma|/T}.
\]
One may discard \(|\gamma|>T(\log X)^2\) with an error smaller
than any fixed negative power of \(X\), by \eqref{eq:total}.
Partition \([1/2,\theta]\) into intervals of length at most
\(1/\log X\), treating the endpoint \(\theta=1/2\) as a single
class. Apply \eqref{eq:density} to the left endpoint of each
interval, with \(U\asymp T(\log X)^2\).
Logarithmic factors, the number of intervals, and
\(q/\varphi(q)\) may all be absorbed in \(X^\eta\).
Together with \eqref{eq:low}, this bounds the zero term in
\eqref{eq:explicit} by
\begin{equation}\label{eq:zerobound}
 \ll_{\theta,\eta}X^\eta\left(
 X^{1/2}B^{1/2}+
 \sup_{1/2\leq\sigma\leq\theta}
 X^\sigma B^{D(\sigma)-1/2}\right).
\end{equation}
Here the factor \(B^{-1/2}\) is precisely the product of the
Gauss-sum factor \(q^{-1/2}\), up to logarithms, and
\(T^{-1/2}\).

For \(1/2\leq\sigma\leq2/3\), one has
\(D(\sigma)-1/2\geq1/4\). Since \(B\ll\sqrt X\), the second
term in \eqref{eq:zerobound} is bounded by the maximum of
\(X^{h(\sigma)}\), where
\[
 h(\sigma)=\sigma-\frac14+\frac{D(\sigma)}2,
 \qquad
 h'(\sigma)=1-\frac{3}{2(2-\sigma)^2}\geq\frac5{32}>0.
\]
Thus the maximum is \(X^{h(\theta)}=X^{a(\theta)}\).
The first term is \(O(X^{3/4})\), and
\(a(\theta)\geq a(1/2)=3/4\).
The last error in \eqref{eq:explicit} is smaller still.
This proves \eqref{eq:local}.
\end{proof}

\subsection{The major-arc model and its Fourier coefficients}

Take \(P=X^{1/3}\). On \(\mathbb R/\mathbb Z\), let \(\MM\)
be the union of the arcs
\[
 \left|\alpha-\frac aq\right|\leq\frac{P}{qX},
 \qquad q\leq P,\quad 0\leq a<q,\quad(a,q)=1.
\]
Distances here are interpreted modulo one. These arcs are
disjoint for all sufficiently large \(X\), since \(2P^2<X\).
Define
\[
 V_X(\alpha)=
 \begin{cases}
 \displaystyle\frac{\mu(q)}{\varphi(q)
       (X^{-1}-2\pi i(\alpha-a/q))},
       &\alpha\text{ lies in the arc about }a/q,\\[6pt]
 0,&\alpha\notin\MM.
 \end{cases}
\]
Lemma~\ref{lem:local} applies on \(\MM\).
For \(\alpha\notin\MM\), Dirichlet approximation supplies
\(q\leq\lfloor\sqrt X\rfloor\) and
\(|\alpha-a/q|\leq1/(q\lfloor\sqrt X\rfloor)\).
For this approximation, either \(q>P\), or
\(qX|\alpha-a/q|>P\). Hence \(B=q(1+X|\alpha-a/q|)>P\),
and the rational main term in \eqref{eq:local} is
\[
 \ll_\eta \frac{X^{1+\eta}}{B}
 \ll_\eta X^{2/3+\eta}.
\]
As \(a(\theta)\geq3/4\), we have the global estimate
\begin{equation}\label{eq:uniform}
 \|F_X-V_X\|_\infty\ll_{\theta,\eta}X^{a(\theta)+\eta}.
\end{equation}

Parseval and an elementary Chebyshev bound give
\begin{equation}\label{eq:L2}
 \|F_X\|_2^2=
 \sum_{m\geq1}\Lambda(m)^2e^{-2m/X}\ll X\log X,
 \qquad
 \|V_X\|_2^2\ll X\sum_{q\leq P}\frac1{\varphi(q)}
 \ll X\log X.
\end{equation}
All norms and integrals on the circle use Lebesgue measure of
total mass one. Therefore
\begin{equation}\label{eq:meansquare}
 \int_0^1|F_X(\alpha)^2-V_X(\alpha)^2|^2\dd\alpha
 \leq \|F_X-V_X\|_\infty^2\|F_X+V_X\|_2^2
 \ll_{\theta,\eta}X^{2a(\theta)+1+2\eta}\log X.
\end{equation}

To identify the model's coefficients, put
\[
 v_X(n)=\int_0^1V_X(\alpha)^2e(-n\alpha)\dd\alpha,
 \qquad
 c_q(n)=\sum_{\substack{0\leq a<q\\(a,q)=1}}e(an/q).
\]
Summing the numerators of each denominator first gives
\begin{equation}\label{eq:v}
 v_X(n)=\sum_{q\leq P}\frac{\mu(q)^2c_q(n)}{\varphi(q)^2}
 \int_{-P/(qX)}^{P/(qX)}
 \frac{e(-n\beta)}{(X^{-1}-2\pi i\beta)^2}\dd\beta.
\end{equation}
For \(n>0\), the full integral over the real line is exactly
\(ne^{-n/X}\). This follows by Fourier inversion applied to the
convolution of \(t\mapsto e^{-t/X}\mathbf1_{t>0}\) with itself.
The absolute value of the omitted tails is \(O(qX/P)\).

Uniformly for \(1\leq n\leq X\), elementary Ramanujan-sum
estimates yield
\begin{equation}\label{eq:ramanujan}
 \sum_{q\leq P}\frac{q|c_q(n)|}{\varphi(q)^2}\ll_\eta X^\eta,
 \qquad
 \sum_{q>P}\frac{|c_q(n)|}{\varphi(q)^2}
 \ll_\eta \frac{X^\eta}{P}.
\end{equation}
For clarity, use \(|c_q(n)|\leq(q,n)\),
\(q/\varphi(q)\ll_\delta q^\delta\), and
\((q,n)=\sum_{d\mid(q,n)}\varphi(d)\).
The first sum then reduces to divisor sums times
\(\sum_{m\leq P/d}m^{-1}\).
For the second, splitting divisors at \(d=P\) gives, for each
\(d\mid n\), a bound \(O_\delta(P^{-1+2\delta})\).
The divisor bound \(d(n)\ll_\delta n^\delta\), with a sufficiently
small \(\delta\), proves \eqref{eq:ramanujan}.

It follows from \eqref{eq:v}--\eqref{eq:ramanujan} that
\begin{equation}\label{eq:mainmodel}
 v_X(n)=ne^{-n/X}\SSS(n)+O_\eta(X^{1+\eta}/P),
 \qquad
 \SSS(n)=\sum_{q\geq1}\frac{\mu(q)^2c_q(n)}{\varphi(q)^2}.
\end{equation}
For even \(n\), its absolutely convergent Euler product is
\begin{equation}\label{eq:singular}
 \SSS(n)=2C_2\prod_{\substack{p\mid n\\p>2}}
                 \frac{p-1}{p-2}\geq2C_2>0,
 \qquad
 C_2=\prod_{p>2}\left(1-\frac1{(p-1)^2}\right).
\end{equation}
Choosing \(\eta<1/3\) in \eqref{eq:mainmodel}, we conclude that
\begin{equation}\label{eq:positive}
 |v_X(n)|\geq cX
 \quad\text{for every even }n\in(X/2,X]
\end{equation}
and all sufficiently large \(X\), for some absolute \(c>0\).

\subsection{Removing prime powers and counting exceptions}

The \(n\)th Fourier coefficient of \(F_X^2\) is
\[
 e^{-n/X}R_\Lambda(n),\qquad
 R_\Lambda(n)=\sum_{u+v=n}\Lambda(u)\Lambda(v).
\]
If \(n\) is not a sum of two primes, every nonzero term in this
last sum involves at least one proper prime power \(p^k\),
\(k\geq2\). There are \(O(\sqrt X\log X)\) such powers up to
\(X\), each contributes at most \((\log X)^2\), and either
summand can be the prime power. Thus, uniformly for such \(n\leq X\),
\begin{equation}\label{eq:primepowers}
 R_\Lambda(n)\ll\sqrt X(\log X)^3=o(X).
\end{equation}
By \eqref{eq:positive} and \eqref{eq:primepowers}, the Fourier
coefficient of \(F_X^2-V_X^2\) has absolute value at least
\(cX/2\) at every exceptional even \(n\in(X/2,X]\), for large \(X\).
Bessel's inequality and \eqref{eq:meansquare} therefore give
\[
 \bigl(E(X)-E(X/2)\bigr)X^2
 \ll_{\theta,\eta}X^{2a(\theta)+1+2\eta}\log X.
\]
Given \(\varepsilon>0\), choose \(\eta>0\) sufficiently small.
This proves
\[
 E(X)-E(X/2)\ll_{\theta,\varepsilon}
 X^{2a(\theta)-1+\varepsilon}.
\]
Since \(2a(\theta)-1\geq1/2\), summing over dyadic intervals
proves Theorem~\ref{thm:main}, with
\(g(\theta)=2a(\theta)-1\).

\subsection{Numerical consequences and the sufficient threshold}

On \([1/2,2/3]\),
\[
 g'(\theta)=2-\frac{3}{(2-\theta)^2}\geq\frac5{16}>0,
 \qquad g(2/3)=\frac7{12}.
\]
Consequently, every fixed \(\theta<2/3\) gives
\(g(\theta)<7/12\). To obtain an actual exponent below \(7/12\)
in \eqref{eq:main}, choose \(\varepsilon\) smaller than this
positive gap. The following are the exponents before adding
the arbitrary \(\varepsilon\).

\begin{center}
\begin{tabular}{@{}ccc@{}}
\toprule
Zero boundary \(\theta\) & Exceptional-set exponent \(g(\theta)\)
 & Decimal value\\
\midrule
\(2/3\)   & \(7/12\)   & \(0.583333\ldots\)\\
\(13/20\) & \(26/45\)  & \(0.577777\ldots\)\\
\(5/8\)   & \(25/44\)  & \(0.568181\ldots\)\\
\(3/5\)   & \(39/70\)  & \(0.557142\ldots\)\\
\(7/12\)  & \(28/51\)  & \(0.549019\ldots\)\\
\(1/2\)   & \(1/2\)    & \(0.500000\)\\
\bottomrule
\end{tabular}
\end{center}

For example, \(\HH_{5/8}\) gives
\[
 E(X)\ll_\varepsilon X^{25/44+\varepsilon},
 \qquad \frac7{12}-\frac{25}{44}=\frac1{66}.
\]
To make the exponent strictly below \(0.55=11/20\), it suffices
to assume a fixed boundary
\[
 \theta<\frac{61-\sqrt{201}}{80}
       =0.5852819140\ldots,
\]
and then take \(\varepsilon\) sufficiently small. Equality in
this last boundary gives \(g(\theta)=0.55\), before the
\(\varepsilon\)-loss.

For comparison, if the density estimate \eqref{eq:density} is
replaced throughout by the total count \eqref{eq:total}, the
same argument gives a local error
\(O_{\theta,\eta}(X^{\theta+1/4+\eta})\) and hence the weaker
bound
\[
 E(X)\ll_{\theta,\varepsilon}X^{2\theta-1/2+\varepsilon}.
\]
That bound beats \(7/12\) only when \(\theta<13/24\).
The use of the fixed-modulus density estimate is what improves
this sufficient threshold to \(2/3\).

\section{Comparison with the Guth--Maynard density estimate}
\label{sec:GM}

Pintz observed that the zero-density improvement of Guth and
Maynard \cite[Theorem~1.2]{GM} does not sharpen the conditional
range used here. Their stated zeta estimate has exponent
\[
 D_{\rm GM}(\sigma)=\frac{15(1-\sigma)}{3+5\sigma},
 \qquad
 N(\sigma,T)\leq T^{D_{\rm GM}(\sigma)+o(1)}.
\]
Comparing it with Ingham's exponent gives, for \(1/2\leq\sigma<1\),
\[
 \frac{15(1-\sigma)}{3+5\sigma}
 <\frac{3(1-\sigma)}{2-\sigma}
 \quad\Longleftrightarrow\quad
 \sigma>\frac7{10}.
\]
The argument in Section~\ref{sec:conditional} uses only
\(1/2\leq\sigma\leq\theta\leq2/3<7/10\).
Thus even an analogous fixed-modulus family bound with this
exponent would not improve \eqref{eq:zerobound} in our range.
Moreover, the zeta estimate by itself does not replace the sum
over all characters modulo \(q\) required in \eqref{eq:density}.
These two points explain why the conditional refinement above
uses the classical estimate.

\section{AI-use declaration and contributions}\label{sec:ai}

ChatGPT, an AI system provided by OpenAI, was used in the development
of this work. Its contribution included proposing and developing
the conditional exponent \(g(\theta)\), checking algebra and
parameter ranges, expanding the circle-method argument, locating
and comparing source statements, and preparing the \LaTeX\ source
and compiled PDF. The conditional part develops the working
argument prepared with Alth\"ofer on 8 October 2026.

Pintz supplied the underlying major-arc method through
\cite{P1,P2} and proposed the structure of the joint paper,
including the combination with \cite{MPR} and \cite{OAI}
and the comparison in Section~\ref{sec:GM}.
Alth\"ofer initiated and directed the collaborative investigation
and the successive refinements of the exceptional-set problem.
The earlier numerical research chain \cite{PA} was also developed
with AI assistance; it is mentioned for context and is not an
input to either proof in this paper.

The use of ChatGPT in this collaboration is separate from the
AI-generated zero-free theorem cited as external work in \cite{OAI}.
The mathematical arguments have not been formalized in Lean.
The scope of the formalization reported for the external \(7/8\)
result is documented separately in \cite{OAILean}; no independent
rebuild of that formalization is claimed here.
The certificates from the earlier numerical chain were not rerun
in preparing this draft. Algebraic checks and successful compilation
are not substitutes for checking the analytic arguments and the
hypotheses of the cited theorems.
The human authors retain responsibility for the final mathematical
content, references, and submission, and this working draft does
not represent their final approval of every statement.

\section*{Acknowledgement}

The first author would like to thank his colleagues and friends
(in alphabetical order) for their help at earlier stages of this work:
S\'andor Krenedits, \'Obuda University;
Alessandro Languasco, Universit\'a di Padova;
M\'at\'e Mr\'az, Szerencsej\'at\'el Plc.;
Szil\'ard Gy.~R\'ev\'esz, HUN-REN Alfr\'ed R\'enyi Institute of Mathematics;
and Gergel\H{o} Szab\'o, HUN-REN Hungarian Research Network.
The second author wants to thank his former postdoc student
Torsten Sillke (now Lufthansa Systems) for regular exchange of ideas.

\bigskip
\noindent\begin{minipage}{\textwidth}
\small
\textbf{J\'anos Pintz}\\
HUN-REN Alfr\'ed R\'enyi Institute of Mathematics\\
Re\'altanoda u.~13--15, H-1053 Budapest, Hungary\\
E-mail: \href{mailto:pintz@renyi.hu}{\nolinkurl{pintz@renyi.hu}}
\end{minipage}

\medskip
\noindent\begin{minipage}{\textwidth}
\small
\textbf{Ingo Alth\"ofer}\\
Friedrich-Schiller University Jena\\
Faculty of Mathematics and Computer Science\\
D-07740 Jena, Germany\\
E-mail: \href{mailto:ingo.althoefer@uni-jena.de}{\nolinkurl{ingo.althoefer@uni-jena.de}}
\end{minipage}
\end{document}